\documentclass{amsart}
\usepackage{amsmath,latexsym,amssymb}
\newcommand{\Aut}{\mathop{\mathrm{Aut}}}

\newtheorem{theorem}{Theorem}[section]
\newtheorem{lemma}[theorem]{Lemma}

\newtheorem{claim}[theorem]{Claim}
\newtheorem*{hauptlemma}{Hauptlemma}
\newtheorem{notation}[theorem]{Hypothesis}
\newtheorem*{TW}{Thompson-Wielandt theorem}
\newcommand{\SL}{\mathop{\mathrm{SL}}}
\newcommand{\soc}{\mathop{\mathrm{soc}}}

\newcommand{\Alt}{\mathop{\mathrm{Alt}}}

\title[Locally primitive graphs of affine type]{An application of the Local $C(G,T)$ Theorem to a conjecture of Weiss}

\author{Pablo Spiga}

\begin{document}
\maketitle
\begin{abstract}
Let $\Gamma$ be a connected $G$-vertex-transitive graph, let $v$ be
a vertex of $\Gamma$ and let $G_v^{\Gamma(v)}$ be the permutation group induced by the action of the vertex-stabiliser $G_v$ on the neighbourhood $\Gamma(v)$. The graph $\Gamma$ is said to be $G$-\emph{locally
  primitive} if $G_v^{\Gamma(v)}$ is primitive.
 
Richard Weiss~\cite{Weisss} conjectured in $1978$ that, there exists a
function $f:\mathbb{N}\to \mathbb{N}$ such that, if $\Gamma$ is a connected
$G$-vertex-transitive locally primitive graph of valency $d$ and $v$ is a vertex of
$\Gamma$ with $|G_v|$ finite, then $|G_v|\leq f(d)$. As an application of the Local $C(G,T)$ Theorem~\cite{CGT}, we prove this
conjecture when $G_v^{\Gamma(v)}$ contains an abelian regular subgroup. In fact, we show that the point-wise stabiliser in $G$ of a ball of $\Gamma$ of radius $4$ is the identity subgroup.
\end{abstract}

\section{Introduction}
Let $\Gamma$ be an undirected graph and let $G$ be a subgroup of the automorphism group $\Aut(\Gamma)$ of $\Gamma$. For each vertex $v$ of $\Gamma$, we let $\Gamma(v)$ denote the set of vertices adjacent to $v$ in $\Gamma$ and $G_v^{\Gamma(v)}$ the permutation group induced on $\Gamma(v)$ by the vertex-stabiliser $G_v$ of $v$. We let $G_v^{[1]}$ denote the subgroup of $G_v$ fixing point-wise $\Gamma(v)$ and we let $G_{uv}^{[1]}=G_u^{[1]}\cap G_{v}^{[1]}$ denote the subgroup of the arc-stabiliser $G_{uv}$ fixing point-wise $\Gamma(u)$ and $\Gamma(v)$. Similarly, given $r\in\mathbb{N}$, we denote by $G_v^{[r]}$ the point-wise stabiliser in $G$ of the ball of $\Gamma$ of radius $r$ centred in $v$ and $G_{uv}^{[r]}=G_u^{[r]}\cap G_v^{[r]}$.

The graph $\Gamma$ is said to be $G$-\emph{vertex-transitive} if $G$ acts
transitively on the vertices of $\Gamma$. We say that the $G$-vertex-transitive graph
$\Gamma$ is \emph{locally primitive} if $G_v^{\Gamma(v)}$ is primitive. In $1978$ Richard
Weiss~\cite{Weiss} conjectured that for a finite connected $G$-vertex-transitive, locally
primitive graph $\Gamma$ and for a vertex $v$ of $\Gamma$, the size of $G_v$ is bounded above by some function depending only on the valency of $\Gamma$ (see also the introduction in~\cite{Weiss2}, where the hypothesis of $\Gamma$ being finite is replaced by the much weaker hypothesis that the order of $G_v$ is finite). The truth of the Weiss Conjecture is still open and
only partial results are known: for example, the case of $G_v^{\Gamma(v)}$ being $2$-transitive has been settled affirmatively with a long series of papers~\cite{T0,T1,T2,T3,T4,Weiss,Weiss0,Weiss1} by the work of Weiss and Trofimov.

A modern method for analysing a finite primitive permutation group is via the O'Nan-Scott theorem. In~\cite[Theorem]{LPS} five types of primitive groups are defined (depending on the group- and action-structure of the socle): HA (\emph{Affine}), AS (\emph{Almost Simple}), SD (\emph{Simple Diagonal}), PA (\emph{Product Action}) and TW (\emph{Twisted Wreath}), and it is shown that  every primitive group belongs to  one of these types. 

The only primitive groups $X$ where the socle $\soc (X)$ of $X$ is a regular normal subgroup are of
Affine type or of Twisted Wreath type. In the former case, $\soc(X)$ is an elementary
abelian $p$-group and, in the latter case, $\soc(X)$ is isomorphic to
the direct product $T^\ell$ with $\ell\geq 2$ and with $T$ a
non-abelian simple group. One indication supporting the Weiss Conjecture is~\cite{Spiga}, where we have shown that  if $\Gamma$ is a connected $G$-vertex-transitive graph of
valency $d$,  $G_v^{\Gamma(v)}$ is a primitive group of TW type and
$(u,v)$ is an arc of $\Gamma$, then $G_{uv}^{[1]}=1$. Another main indication supporting the Weiss Conjecture is in~\cite{Weiss}, where Weiss
shows that, if $\Gamma$, $d$, $u$ and $v$ are as above and $G_v^{\Gamma(v)}$ is a primitive group of HA type, then either $G_{uv}^{[1]}=1$ or the socle of $G_v^{\Gamma(v)}$ is an elementary
abelian $2$- or $3$-group (more information on the structure of $G_v^{\Gamma(v)}$ in
the latter case is given in~\cite[Theorem~$(i)$]{Weiss}). 

In this paper, as an application of the Local $C(G,T)$ Theorem, we completely settle the case of $G_v^{\Gamma(v)}$ containg an abelian regular subgroup, that is,  of HA type.

\begin{theorem}\label{thm}Let $\Gamma$ be a connected
  $G$-vertex-transitive graph and let $(u,v)$ be  an arc of
  $\Gamma$. Suppose that  $G_v^{\Gamma(v)}$ is a primitive group containg an abelian regular subgroup and let $|\Gamma(v)|=r^\ell$ for some prime $r$ and positive integer $\ell$. Then $G_{uv}^{[1]}=1$ when $r\geq 5$, $G_{uv}^{[2]}=1$ when $r=3$ and $G_{uv}^{[3]}=1$ when $r=2$.
In particular,  $G_{v}^{[4]}=1$. 
\end{theorem}

Other recent and significant indications
towards a positive solution to the Weiss Conjecture  are given in~\cite{Giu1,Giu2,MSV,PoSV,PSV,PPSS}.

\subsection*{Acknowledgements}I am in debt (and I will always be) to Bernd Stellmacher. In fact, it was Bernd that pointed out to me the relevance of the Local $C(G,T)$ Theorem for a proof of Theorem~\ref{thm}. Actually, the ideas and the proof of Theorem~\ref{thm} are based on conversations and a manuscript of Bernd. I am also in debt to Luke Morgan for reading a preliminary version of this paper.  To both Bernd and Luke, my sincere thanks.

\section{Preliminaries}\label{preliminaries}

All graphs considered in this paper are simple (without multiple edges and without loops), undirected and connected. Given a subgroup $G$ of the automorphism group $\Aut(\Gamma)$ of $\Gamma$ and a vertex $v$ of $\Gamma$, we denote by $G_v$ the stabiliser of $v$ in $G$, that is, $G_v:=\{g\in G\mid v^g=v\}$. Moreover, we let $G_v^{\Gamma(v)}$ be the permutation group induced by the action of $G_v$ on the neighbourhood $\Gamma(v)$ of $v$. 

In what follows we let $\Gamma$ be a connected graph and $G$ be a subgroup of $\Aut(\Gamma)$ with $G_v^{\Gamma(v)}$ transitive for every vertex $v$ of $\Gamma$. We start by recalling some basic facts.

\begin{lemma}\label{lemma:21}
The group $G$ acts transitively on the edges of $\Gamma$.
\end{lemma}
\begin{proof}
Given two vertices $v,v'$ of $\Gamma$, denote by $d(v,v')$ the length of a shortest path  from $v$ to $v'$.
Let $\{u,v\}$ and $\{w,t\}$ be two edges of $\Gamma$. We argue by induction on $$d:=\min\{d(u,w),d(u,t),d(v,w),d(v,t)\}. $$
Without loss of generality we may assume that $d=d(u,w)$. Let $u=u_0,u_1,\ldots,u_d=w$ be a path of length $d$ in $\Gamma$ from $u$ to $w$. If $d=0$, then $u=w$ and, as $G_u^{\Gamma(u)}$ is transitive, there exists $x\in G_u$ with $v^x=t$. Thus, $\{u,v\}^x=\{w,t\}$. Suppose that $d>0$. Since $G_w^{\Gamma(w)}$ is transitive, there exists $x\in G_w$ with $u_{d-1}^x=t$ and hence $\{w,u_{d-1}\}^x=\{w,t\}$. As $d(u,u_{d-1})=d-1$, by induction, there exists $g\in G$ with $\{u,v\}^g=\{w,u_{d-1}\}$ and hence $\{u,v\}^{gx}=\{w,t\}$.
\end{proof}

\begin{hauptlemma}%\label{lemma:54}
Let $\{u,v\}$ be an edge of $\Gamma$ and let $N\leq G_{u}\cap G_{v}$ with
$$
({\bf N}_{{G_{w}}}( {N}))^{\Gamma(w)}\quad\textrm{transitive for each }w\in
\{u,v\}.
$$
Then $N=1$.
\end{hauptlemma}
\begin{proof}
Write $K:={\bf N}_G (N)$. By hypothesis, $K_u^{\Gamma(u)}$ and $K_v^{\Gamma(v)}$ are both transitive. Thus, from Lemma~\ref{lemma:21}, $K$ acts transitively on the edges of $\Gamma$. As $N\unlhd K$ and $N$ fixes the arc $(u,v)$, we see that $N$ fixes the end points of every edge of $\Gamma$. 
\end{proof}

\begin{notation}\label{notation1}{\rm In what follows we let $\Gamma$ be a connected graph,  $G$ be a group of automorphisms of $\Gamma$ and $\{u,v\}$ be an edge of $\Gamma$. We assume that
\begin{enumerate}
\item $G$ is transitive on the vertices of $\Gamma$;
\item $G_v$ is finite;
\item $G_v^{\Gamma(v)}$ is primitive.
\end{enumerate}
We let $G_v^{[1]}$ denote the subgroup of $G_v$ fixing point-wise $\Gamma(v)$ and we let $G_{uv}^{[1]}=G_u^{[1]}\cap G_{v}^{[1]}$ denote the subgroup of the arc-stabiliser $G_{uv}$ fixing point-wise $\Gamma(u)\cup\Gamma(v)$.  }
\end{notation}

We now recall the Thompson-Wielandt Theorem with its formulation as  in~\cite{Van}, see also~\cite{Spiga1}. (Here, ${\bf F}^*(X)$ denotes the generalised Fitting subgroup.)

\begin{TW}The group $G_{uv}^{[1]}$ is a $p$-group for some prime $p$. Moreover, either $G_{uv}^{[1]}=1$ or  ${\bf F}^*(G_v)={\bf O}_p (G_v)$ and ${\bf F}^*(G_{uv})={\bf O}_p ({G_{uv}})$.
\end{TW}
\begin{notation}\label{notation2}{\rm Together with Hypothesis~\ref{notation1} we also assume that $G_{uv}^{[1]}\neq 1$ and we let $p$ be the prime number with ${\bf F}^* (G_v)={\bf O}_p ({G_v})$ and ${\bf F}^*(G_{uv})={\bf O}_p ({G_{uv}})$. Observe that $G_{uv}^{[1]}$ is a non-trivial $p$-group.

For simplicity, given a vertex $w$ of $\Gamma$, we write $$Q_w:={\bf O}_p ({{G_w^{[1]}}}),\quad L_w:=\langle Q_wQ_t\mid t\in \Gamma(w)\rangle,\quad T:=Q_uQ_v.$$}
\end{notation}

We immediately deduce the following lemma.
\begin{lemma}\label{basbas}We have
${\bf F}^* ({G_v^{[1]}})={\bf F}^* (G_v)=Q_v$.
\end{lemma}
\begin{proof}
Suppose that ${\bf O}_p( {G_v})\nleq G_v^{[1]}$. Since $G_v^{\Gamma(v)}$ is primitive, $({\bf O}_p ({G_v}))^{\Gamma(v)}$ is a transitive $p$-group and hence $G_v^{\Gamma(v)}$ is a primitive group containing an abelian regular subgroup and $({\bf O}_p ({G_v}))^{\Gamma(v)}$ is the socle of $G_v^{\Gamma(v)}$. Set $V:=({\bf O}_p ({G_v}))^{\Gamma(v)}$ and $H:=(G_{uv})^{\Gamma(v)}$. 

The elementary abelian $p$-group $V$ is an irreducible $H$-module via 
the action 
of $H$ on $V$ by conjugation. If ${\bf O}_p (H)\neq 1$, then 
${\bf C}_V\left( {\bf O}_p ( H) \right)$ is a non-trivial proper 
$H$-submodule of $V$, a contradiction. Thus ${\bf O}_p (H)=1$. Hence 
$({\bf O}_p ({G_{uv}}))^{\Gamma(v)}=1$. Therefore 
${\bf O}_p ({G_{uv}})\leq G_v^{[1]}$ and 
 ${\bf O}_p ({G_{uv}})={\bf O}_p ({G_v^{[1]}})$. This contradicts the 
Hauptlemma applied with $N:={\bf O}_p ({G_{v}^{[1]}})$. Therefore 
${\bf O}_p ({G_v})\leq G_v^{[1]}$ and hence 
${\bf O}_p ({G_v})={\bf O}_p ({G_v^{[1]}})=Q_v$.

The rest of the lemma follows from the Thompson-Wielandt Theorem.
\end{proof}

\begin{lemma}\label{basicL}The permutation group $L_v^{\Gamma(v)}$ is transitive and  $[G_v^{[1]},L_v]\leq Q_v$.
\end{lemma}
\begin{proof}
As $L_v\unlhd G_v$ and $G_v^{\Gamma(v)}$ is primitive, we have that either $L_v^{\Gamma(v)}$ is transitive or $L_v\leq G_v^{[1]}$. Suppose that $L_v\leq G_v^{[1]}$. Then $Q_u\leq G_v^{[1]}$. Thus $Q_u\unlhd G_v^{[1]}$ and $Q_u\leq \mathbf O_p( {{G_v^{[1]}}})=Q_v$. Therefore $Q_u=Q_v$. Now the Hauptlemma applied with $N:=Q_v$ gives $Q_v=1$. From Lemma \ref{basbas}, we get ${\bf F}^*(G_v)=1$, a contradiction.

As $Q_v\unlhd G_v^{[1]}$, we have $[G_v^{[1]},Q_v]\leq Q_v$. Now, let $w\in \Gamma(v)$. As $G_v^{[1]}$ normalises $Q_w$, we get $[G_v^{[1]},Q_w]\leq G_v^{[1]}\cap Q_w\leq \mathbf O_p ({{G_v^{[1]}}})=Q_v$. Now the second part of the  lemma immediately follows from the definition of $L_v$. 
\end{proof}

\begin{lemma}\label{lemma:43}Let $C$ be a non-identity characteristic subgroup of $T$. Then ${\bf N}_{{G_w}}(C)=G_{uv}$, for each $w\in \{u,v\}$.
\end{lemma}
\begin{proof}
Let $w\in \{u,v\}$ and write $N:={\bf N}_{{G}}({C})$. As $C\unlhd G_{uv}$, we have $G_{uv}\leq N_w$ and, by maximality, either $N_w=G_w$ or $N_w=G_{uv}$.

Assume that $N_w=G_w$. As $G$ acts transitively on the arcs of $\Gamma$, there exists $t\in {\bf N}_G({{G_{uv}}})$ with $(u,v)^t=(v,u)$. Hence $t$ normalises $Q_u Q_v=T$ and thus also $C$. This gives $N_u=G_u$ and $N_v=G_v$ and the Hauptlemma  yields $C=1$.
\end{proof}

\begin{lemma}\label{lemma:44}Suppose that $T\in {\rm Syl}_p({{L_v}})$. Then there exist subnormal subgroups $E_1,\ldots,E_r$ of $G_v$ such that for $E:=\langle E_1,\ldots,E_r\rangle$ the following hold:
\begin{description}
\item[(a)]$E\leq L_v$ and $G_v$ acts transitively on $\{E_1,\ldots,E_r\}$; in particular $E\unlhd G_v$.
\item[(b)]$[E_i,E_j]=1$, for $i,j\in \{1,\ldots,r\}$ with $i\neq j$.
\item[(c)]$G_v=EG_{uv}$.
\item[(d)]$\mathbf O_p( {{E_i}})=[\mathbf O_p({{E_i}}),E_i]=[Q_v,E_i]$ and $E_i={\bf O}^p({E_i})$, for $i\in \{1,\ldots,r\}$.
\item[(e)] $[E_i, \Omega_1({\bf Z}({{T}}))] \neq 1$, for $i\in \{1,\ldots,r\}$.
\item[(f)]For each $i\in \{1,\ldots,r\}$, one of the following holds:
\begin{description}
\item[(i)]$E_i/\mathbf O_p({{E_i}})\cong (\SL_2(q))'$ with $q=p^n$ for some $n\geq 1$, $\mathbf O_p( {{E_i}})=\Omega_1({\bf Z}({{\mathbf O_p({{E_i}})}}))$ and $\mathbf O_ p ({{E_i}})/{\bf Z}({{E_i}})$ is a natural $(\SL_2(q))'$-module for $E_i/\mathbf O_p( {{E_i}})$,
\item[(ii)]$p=3$, $E_i/\mathbf O_3( {{E_i}})\cong (\SL_2(q))'$ with $q=3^n$ for some $n\geq 1$, ${\bf Z}({{E_i}})=\Phi(\mathbf O_3( {{E_i}}))=(\mathbf O_3({{E_i}}))'$, $|{\bf Z}({{E_i}})|=q$, and $\mathbf O_3( {{E_i}})/\Omega_1({\bf Z}({{\mathbf O_3({{E_i}})}}))$ and $\Omega_1({\bf Z}({{\mathbf O_3({{E_i}}})}))/{\bf Z}({{E_i}})$ are both natural $(\SL_2(q))'$-modules for $E_i/\mathbf O_3( {{E_i}})$, 
\item[(iii)]$p=2$, $E_i/\mathbf O_2({{E_i}})\cong \Alt(2^n+1)$ for some $n\geq 2$, $\mathbf O_2({{E_i}})=\Omega_1({\bf Z}({{\mathbf O_2({{E_i}})}}))$, and $\mathbf O_2( {{E_i}})/{\bf Z}({{E_i}})$ is a natural $\Alt(2^n+1)$-module for $E_i/\mathbf O_2( {{E_i}})$.
\end{description}
\end{description}
\end{lemma}
\begin{proof}
Recall that, according to~\cite[Definition~$1.2$]{CGT}, a finite group $X$ is said to be of characteristic $p$ if ${\bf C}_X\left({{\mathbf O_p (X)}}\right)\leq \mathbf O_p (X)$. 
By Lemma~\ref{basbas}, we have $Q_v={\bf F}^*(G_v )={\bf F}^* ({G_v^{[1]}})$ and, in particular, $Q_v=\mathbf O_p ({{L_v}})={\bf F}^*(L_v)$ and $L_v$ is of characteristic $p$. Moreover, by Lemma~\ref{lemma:43}, we have $$\langle {\bf N}_{{L_v}}({C})\mid C\neq 1,\,C\textrm{ characteristic in }T \rangle=G_{uv}\cap L_v<L_v$$ because $L_v^{\Gamma(v)}$ is transitive by Lemma~\ref{basicL}. Now, the group $L_v$ satisfies the hypothesis of~\cite[Corollary~$1.6$]{CGT} and thus the proof follows immediately from the Local $C(G,T)$ Theorem. 

Parts~(b) and~(c) follow from~\cite[Theorem~$1.5$~(b) and~(c)]{CGT} and the fact that $G_{uv}$ contains the subgroup $C(L_v,T)$ defined in~\cite{CGT}. 

From~\cite[Theorem~$1.5$~(a)]{CGT}, we have $\{E_1,\ldots,E_r\}^{G_v}=\{E_1,\ldots,E_r\}$ and hence $E\unlhd G_v$. Using Part~{\bf (c)} choose $i\in \{1,\ldots,r\}$ with $E_i\nleq G_{uv}$ and set $X=\langle E_i^g\mid g\in G_v\rangle$. Observe that $X\unlhd G_v$ and hence the primitivity of $G_v^{\Gamma(v)}$ yields $G_v=XG_{uv}$. In particular, replacing the family $\{E_1,\ldots, E_r\}$ and the group $E$ by the family $\{E_i^g\mid g\in G_v\}$ and the group $X$, we may assume that also Part~{\bf (a)} holds.

Part~(e) and the equalities $\mathbf O_p(  {{E_i}})=[\mathbf O_p( {{E_i})}, E_i]$ and $E_i{=\bf O}^p({E_i})$ in Part~(d) follow  from~\cite[Definition~$1.4$~(i)]{CGT}. Now, the equality $\mathbf O_p( {{E_i}})=[Q_v,E_i]$ in Part~(d) follows from $Q_v=\mathbf O_p( {L_v})$ and $E_i=\mathbf O^p(E_i)$. Finally, Part~(f) follows from~\cite[Definition~$1.4$]{CGT}.
\end{proof}
\begin{lemma}\label{lemma:45}
If $T\in{\rm Syl}_p({{L_v}})$, then $G_{uv}^{[2]}=1$ if $p\neq 3$ and $G_{uv}^{[3]}=1$ if $p=3$.
\end{lemma}
\begin{proof}
Assume the notation in Lemma~\ref{lemma:44}. Since $G_{uv}^{[1]}$ is a $p$-group, we see that $G_v^{[2]}$ is a $p$-group, and hence $G_v^{[2]}\leq Q_v$. If $G_{uv}^{[2]}=1$, then the lemma follows immediately and hence we assume that $G_{uv}^{[2]}\neq 1$.

For each $i\in \{1,\ldots,r\}$, we set 
$$V_i:=\Omega_1({\bf Z}({{\mathbf O_p ({{E_i}})}})).$$

For each vertex $w$ of $\Gamma$, we set $$Z_w:=\Omega_1({\bf Z}( {Q_w})).$$  

Now we subdivide the proof into six claims from which the proof will immediately follow.

\begin{claim}
\label{ location of zb and [zb,ei]}
$Z_u \leq Q_v$.
\end{claim}

\noindent Suppose that $Z_u\nleq Q_v$. 
Observe that $G_{v}^{[2]}=G_v^{[2]}\cap Q_v\leq G_u^{[1]}\cap Q_v\leq \mathbf O_p ({{G_u^{[1]}}})=Q_u$. 
Hence $Z_u$ centralises $G_v^{[2]}$. 
Thus also $H_v:=\langle Z_u^x\mid x\in G_v\rangle$ centralises $G_v^{[2]}$. 
Observe that, as $Z_u\nleq Q_v$, the group $H_v$ acts transitively on $\Gamma(v)$. 

By arc-transitivity we have $Z_v\nleq Q_u$, and hence a symmetric argument gives that $H_u:=\langle Z_v^x\mid x\in G_u\rangle$ centralises $G_u^{[2]}$ and acts transitively on $\Gamma(u)$. 

We conclude that $G_{uv}^{[2]}=G_u^{[2]}\cap G_v^{[2]}$ is centralised by $H_v$ and $H_u$. The Hauptlemma yields $G_{uv}^{[2]}=1$, a contradiction.
% 
%Now the second part of the claim follows from the first part by Lemma~\ref{lemma:44}~(d).
~$_\blacksquare$

\begin{claim}
\label{[vi,ei] in za}
$[V_i,E_i] \leq Z_v$ for every $i\in \{1,\ldots,r\}$.
\end{claim}

\noindent Let $i\in \{1,\ldots,r\}$. Observe that $[Z_v,E_i] \leq [Q_v, E_i] = \mathbf O_p({E_i})$ and hence $Z_v$ and $E_i$ normalise each other. Now $Z_v \cap E_i$ is an elementary abelian normal $p$-subgroup of $E_i$, whence $Z_v \cap E_i \leq  \mathbf O_p({E_i})$. Since $\mathbf O_p({E_i})=[Q_v,E_i]   \leq Q_v \cap E_i$ and $Z_v$ is central in $Q_v$, we have $Z_v \cap E_i \leq \Omega_1({\bf Z}({\mathbf O_p({E_i})})) = V_i$. This shows $$Z_v \cap E_i = Z_v \cap V_i.$$
Since $V_i / \mathbf Z(E_i)$ is a simple $(E_i/\mathbf O_p( {E_i}))$-module and $Z_v \cap V_i$ is $E_i$-invariant, we have either $Z_v \cap V_i \leq {\bf Z}({E_i})$ or $V_i = (Z_v\cap V_i){\bf Z}({E_i})$. In the latter case, we have 
$$[V_i,E_i] = [(Z_v\cap V_i ){\bf Z}( {E_i}), E_i] = [Z_v\cap V_i,E_i] \leq Z_v$$
and the claim follows. In the former case, since ${\bf O}^p(E_i)=E_i$, we see that
$$ [Z_v,E_i] = [Z_v,E_i,E_i] \leq [Z_v\cap E_i,E_i]=[Z_v\cap V_i, E_i] \leq [{\bf Z}({E_i}),E_i] = 1$$
and hence $E_i$ centralises $Z_v$. From Lemma~\ref{basbas}, we have $Q_v={\bf F}^* (G_v)$ and hence ${\bf F}^*(L_v)=Q_v$ and 
${\bf C}_{{L_v}}({{Q_v}})\leq Q_v$. 
As $Q_v\leq T$, we get 
${\bf C}_{{L_v}}({{T}})$
$\leq {\bf C}_{{L_v}}({{Q_v}})\leq Q_v$. 
From this it follows that  
$\Omega_1({\bf Z}({T})) \leq Z_v$, but this  contradicts Lemma~\ref{lemma:44}(e).~$_\blacksquare$

\begin{claim}
\label{ [zb, ei] notin vi}
$[Z_u,E_i] \nleqslant V_i$ for every $i\in \{1,\ldots,r\}$.
\end{claim}

\noindent Suppose that $[Z_u,E_i] \leq V_i$ for some $i\in \{1,\ldots,r\}$. Let $U := \langle Z_u^x\mid x\in E_i \rangle$. Then $U$ is normalised by $E_i$ and is a  subgroup of $Q_v$ by Claim~\ref{ location of zb and [zb,ei]}. We have
$$[U,E_i] = [\langle Z_u^x\mid x\in E_i\rangle ,E_i]  = \langle [Z_u,E_i]^x\mid x\in {E_i} \rangle \leq \langle V_i^x\mid x\in {E_i} \rangle = V_i.$$
As $E_i=\mathbf O^p(E_i)$,  using Claim~\ref{[vi,ei] in za}, we obtain
$$[U,E_i] = [U,E_i,E_i] \leq  [V_i,E_i] \leq Z_v.$$
 Thus $[Z_u,E_i] \leq [U,E_i] \leq Z_v$, which shows that $E_i$ normalises $Z_u Z_v$. By Lemma~\ref{lemma:44}~(a) and~(c), $Z_u Z_v$ is normalised by $\langle E_i, G_{uv}\rangle = G_v$.  It follows that $Z_u Z_v$ is normalised by
$\langle G_v, G_{\{u, v \}} \rangle = G,$ a contradiction.~$_\blacksquare$

\begin{claim}\label{ op(ei) non-abelian}$p=3$ and $\mathbf O_p(E_i)$  is non-abelian for every $i\in \{1,\ldots,r\}$. 
\end{claim}

\noindent By Claim~\ref{ location of zb and [zb,ei]}, we have $[Z_u,E_i]\leq [Q_v,E_i]\leq \mathbf O_p ({E_i})$ and by Claim~\ref{ [zb, ei] notin vi}, we have $[Z_u,E_i]\nleq V_i$. Thus $V_i \neq \mathbf O_p(E_i)$ and hence only case (ii) of Lemma~\ref{lemma:44}~(f) can hold. Thus $p=3$ and $(\mathbf O_p({E_i}))' \neq 1$.~$_\blacksquare$

\begin{claim}
\label{ zbg not qb} 
For every $i\in \{1,\ldots,r\}$, there exists $g \in E_i$ with $Z_u^g \nleq Q_u$.
\end{claim}

\noindent Let $i\in \{1,\ldots,r\}$. Suppose that $Z_u^g\leq Q_u$ for every $g\in E_i$ and set $U:= \langle Z_u^x\mid x\in {E_i} \rangle$. 

Since $Z_u$ centralises $Q_u$, we have $Z_u \leq {\bf Z}( U)$.  Now $U$ is $E_i$-invariant and moreover, since $Z_u$ is contained also in $Q_v$ by  Claim~\ref{ location of zb and [zb,ei]}, we get that $U$ is contained in $Q_v$. Therefore $[U,E_i] \leq [Q_v,E_i] =\mathbf O_p({E_i})$ (where in the last equality we used Lemma~\ref{lemma:44}~(d)). 

Note that $[Z_u,E_i] \leq [ {\bf Z}( U), E_i]$ and hence $[{\bf Z}( U),E_i] \nleq V_i$ by Claim~\ref{ [zb, ei] notin vi}. Since $\mathbf O_p(E_i)/ V_i$ is a simple $(E_i/\mathbf O_p( {E_i}))$-module by Lemma~\ref{lemma:44}~(f)~(ii), we have $[\mathbf Z(U),E_i] V_i = \mathbf O_p(E_i)$. Now $V_i \leq \mathbf Z(\mathbf O_p(E_i))$ and $[\mathbf Z(U),E_i] \leq \mathbf Z(U)$ since $E_i$ normalises $U$. In particular, $V_i$ and $[\mathbf Z(U),E_i]$ are both abelian and centralise each other. Thus $\mathbf O_p(E_i) = [\mathbf Z(U),E_i]V_i$ is abelian, a contradiction to Claim~\ref{ op(ei) non-abelian}.~$_\blacksquare$

\smallskip

Let $i\in \{1,\ldots,r\}$ and let $g\in E_i$  with $Z_u^g\nleq Q_u$. Now set $$w := u^g,\quad X:=\langle Z_{w}^x\mid x\in G_u\rangle.$$

\begin{claim}\label{alst}
 $X$  acts transitively on $\Gamma(u)$.
\end{claim}

\noindent Since $X\unlhd G_u$, it suffices to show that $X\nleq G_u^{[1]}$. Observe that $Z_w\leq Q_v$ by Claim~\ref{ location of zb and [zb,ei]}.
If $X\leq G_u^{[1]}$, then $Z_{w} \leq G_u^{[1]} \cap Q_v \leq Q_u$, a contradiction to our choice of $w$.~$_\blacksquare$

\smallskip

Recall that $G_{vu}^{[1]}$ is a $p$-group and hence so is $G_{vw}^{[1]}$. Since $G_{vw}^{[1]}$ is normalised by $G_w^{[1]}$, we get $G_{vw}^{[1]}\leq \mathbf O_p\left( {G_w^{[1]}}\right)=Q_w$. As $u$ and $w$ are both neighbours of $v$, we have $G_u^{[3]}\leq G_{vw}^{[1]}$. Therefore $G_u^{[3]}\leq Q_w$ and hence $Z_w$ centralises $G_u^{[3]}$. As $G_u^{[3]}$ is $G_u$-invariant,  $X$ centralises $G_u^{[3]}$ and hence also $G_{u}^{[3]}\cap G_{v}^{[3]}=G_{uv}^{[3]}$. The arc-transitivity and Claim~\ref{alst} give $({\bf N}_{{G_t}}({{G_{uv}^{[3]}}}))^{\Gamma(t)}$ is transitive, for $t\in \{u,v\}$. The Hauptlemma  gives $G_{uv}^{[3]}=1$.
%
%$$G_u^{[4]}\leq G_v^{[3]} \leq G_v^{[2]} \cap Q_{w},$$
%the group $Z_w$ centralises $G_u^{[4]}$. As $G_u^{[4]}$ is $G_u$-invariant,  $X$ centralises $G_u^{[4]}$. Now $G_v^{[5]}\leq G_u^{[4]}$ and hence $\nor{{G_t}}{{G_v^{[5]}}}^{\Gamma(t)}$ is transitive, for $t\in \{u,v\}$. The Hauptlemma  gives $G_v^{[5]}=1$.
\end{proof}

\section{Proof of Theorem~\ref{thm}}

\begin{proof}[Proof of Theorem~$\ref{thm}$]
Let $\Gamma$ be  a $G$-vertex-transitive graph and let $v$ be a vertex of $\Gamma$. Suppose that $G_v^{\Gamma(v)}$ is a primitive group containing an abelian regular subgroup and write $|\Gamma(v)|=r^\ell$ for some prime $r$ and some $\ell\geq 1$. 

Let $u$ be a neighbour of $v$. If $G_{uv}^{[1]}=1$, then  there is nothing to prove. Assume then that $G_{uv}^{[1]}\ne 1$. In particular, the hypotheses in Hypotheses~\ref{notation1} and~\ref{notation2} apply to $G$, $\Gamma$ and $\{u,v\}$. Now, we adopt the terminology  in Hypothesis~\ref{notation2}.

Let $N$ be a normal subgroup of $G_v$ minimal (with respect to set inclusion) subject to $Q_v\leq N\leq L_v$ and $N\nleq G_v^{[1]}$. Observe that $N$ is well-defined because $Q_v\leq L_v\unlhd G_v$ and $L_v\nleq G_v^{[1]}$ by Lemma~\ref{basicL}. As $G_v^{\Gamma(v)}$ is primitive containing an abelian regular subgroup, $N^{\Gamma(v)}$ is the socle of $G_v^{\Gamma(v)}$ and is the unique normal elementary abelian regular $r$-subgroup of $G_v^{\Gamma(v)}$.

Now, the definition of $L_v$ and the transitivity of $N^{\Gamma(v)}$ gives $L_v=Q_v\langle Q_w\mid w\in \Gamma(v)\rangle=Q_v\langle Q_u^n\mid n\in N\rangle\leq Q_vQ_uN=TN\leq L_v$. Thus $L_v=NT$.

If $r=p$, then $L_v^{\Gamma(v)}$ is a $p$-group  and the primitivity of $G_v^{\Gamma(v)}$ gives $L_v=N$. Therefore $T\leq G_v^{[1]}$ and hence $Q_v=Q_u$. Now, the Hauptlemma yields $Q_v=Q_u=1$, a contradiction. Thus $r\neq p$.

Now, $N^{\Gamma(v)}\cong N/(N\cap G_v^{[1]})$ is an $r$-group. Moreover, by Lemma~\ref{basicL}, we have $[N\cap G_v^{[1]},L_v]\leq Q_v$, that is, $L_v$ centralises $(N\cap G_v^{[1]})/Q_v$. This shows that $N/Q_v$ is nilpotent and hence the minimality of $N$ gives that $N/Q_v$ is an $r$-group. Therefore $T\in {\rm Syl}_p ({L_v})$. Now, the hypothesis of Lemma~\ref{lemma:44} are satisfied. We adopt the notation in Lemma~\ref{lemma:44}.

Since $N$ and $T$ are soluble, so is $L_v$. Therefore $(\SL_2(q))'$ can be a section of $L_v$ only when $q\in \{2,3\}$. We deduce $(p,r)\in \{(2,3),(3,2)\}$. In particular, $r\leq 3$ and $p\leq 3$. Now  the proof follows from Lemma~\ref{lemma:45}.
\end{proof}

\end{document}